\documentclass[12pt]{article}  

\usepackage[width=17.5cm,height=23cm]{geometry}
\usepackage{amsmath,amssymb,amsthm,latexsym}
\usepackage{stmaryrd}
\usepackage{tikz}
\usepackage{tikz-cd}

\newcommand{\E}{\mathsf{E}}

\newcommand{\KGL}{\mathsf{KGL}}

\newcommand{\KQ}{\mathsf{KQ}}

\newcommand{\GW}{\mathsf{GW}}
\newcommand{\Th}{\mathsf{Th}}

\newcommand{\Z}{\mathbb{Z}}
\newcommand{\N}{\mathbb{N}}

\newcommand{\R}{\mathbb{R}}
\newcommand{\FF}{\mathbb{F}}

\newcommand{\unit}{\mathbf{1}}
\newcommand{\SH}{\mathbf{SH}}
\newcommand{\A}{\mathbf{A}}
\renewcommand{\P}{\mathbf{P}}
\newcommand{\HGr}{\mathbf{HGr}}
\newcommand{\HP}{\mathbf{HP}}

\newcommand{\SL}{\mathrm{SL}}
\newcommand{\Sing}{\mathrm{Sing}}
\newcommand{\KSp}{\mathrm{KSp}}
\newcommand{\Spec}{\mathrm{Spec}}

\newcommand{\Forg}{\mathsf{Forg}}
\newcommand{\Hyper}{\mathsf{Hyp}}

\let\smash=\wedge
\author{K.~Arun Kumar and Oliver R\"ondigs}
\title{A cellular absolute motivic ring spectrum representing Hermitian $K$-theory}

\newtheorem{theorem}{Theorem}[section]
\newtheorem*{theorem*}{Theorem}
\newtheorem{lemma}[theorem]{Lemma}

\newtheorem*{proposition*}{Proposition}

\newtheorem*{corollary*}{Corollary}

\newtheorem{defn*}{Definition}

\theoremstyle{definition}
\newtheorem{remark}[theorem]{Remark}
\newtheorem{remark*}{Remark}

\begin{document}
\maketitle
\begin{abstract}
  In the Morel-Voevodsky motivic stable homotopy category of a quasi-compact quasi-separated scheme $S$, several candidates exist for a motivic spectrum representing hermitian $K$-theory.
  This note shows that the cellular absolute motivic spectrum constructed
  in \cite{kumar.thesis} via the geometry of orthogonal
  and hyperbolic Grassmannians over $S$ coincides with the
  motivic ring spectrum constructed in \cite{chn}.
\end{abstract}

\section{Introduction}

With his creation of higher algebraic $K$-theory, Daniel Quillen 
constructed an interesting cohomology theory of schemes. The
Morel-Voevodsky $\P^1$-stable $\A^1$-homotopy theory provides
a natural home for this cohomology theory over regular
schemes, similar to the classical stable homotopy category hosting
complex topological $K$-theory.
More precisely, higher algebraic $K$-groups
of any smooth $S$-scheme $X$ are
represented by a motivic ring spectrum $\KGL_S$ in the
motivic stable homotopy category $\SH(S)$ 
for any Noetherian regular
scheme $S$ \cite{voevodsky.icm}, \cite{ppr}, \cite{rso.strict}.
Questions around algebraic $K$-theory, such as the
existence of the motivic spectral sequence analogous to the Atiyah-Hirzebruch
spectral sequence for complex topological $K$-theory, were a major driving force
for motivic stable homotopy theory.

The motivic spectrum $\KGL_S$
is constructed via the infinite Grassmannian, the union of the finite Grassmannians of $d$-dimensional linear subspaces of a $c+d$-dimensional affine space $\A^{c+d}_S$ over $S$.
Since these are built in a suitable sense out of algebraic spheres over $S$, the
motivic spectrum $\KGL_S$ is cellular \cite{dugger-isaksen.cell}.
To explain the last adjective in the title, the
collection $\{\KGL_S\}$ of motivic spectra is absolute in the sense that
it pulls back along any morphism, that is, $f^\ast \KGL_S\simeq \KGL_R$ for any
morphism $f\colon R\to S$.

Hermitian $K$-theory, as conceived by Karoubi, adds the complexity of
suitable bilinear forms
to the vector bundles underlying algebraic $K$-theory. 
From the topological perspective, it is related to algebraic $K$-theory
in the same way as real topological
$K$-theory is related to  complex topological $K$-theory.
Under the additional restriction that $2$ be invertible in the regular affine
Noetherian base scheme $S$,
a motivic spectrum in $\SH(S)$ representing Karoubi's hermitian $K$-theory
(also known as higher Grothendieck-Witt theory)
was provided in \cite{hornbostel}. However, the restriction on the
invertibility of $2$ prevents its absoluteness, and several authors have
worked on extending the theory, in particular to the terminal scheme
$\Spec(\Z)$  \cite{zbMATH07124513}, \cite{spitzweck.herm}.

The first author, in a 2020 PhD thesis \cite{kumar.thesis}
supervised by the second author, obtained a cellular absolute
motivic spectrum
$\KQ_S^\prime$ over any scheme via a geometric description. 
The essential parts of the thesis have been published in \cite{kumar}.
If $2$ is invertible in the Noetherian regular scheme
$S$, it represents Karoubi's hermitian $K$-theory.
However, a ring structure, obtained previously in \cite{lopez.thesis},
as well as several structural properties, were missing, as well as an
idenfication of what this motivic spectrum represents over regular
Noetherian bases in which $2$ is not invertible.

In 2024, \cite{chn} provided an absolute motivic ring spectrum $\KQ_S$
over any quasi-compact quasi-separated scheme $S$, with
further desirable properties, such as absolute purity and twisted Thom
isomorphisms. Their work is based on further deep results on hermitian $K$-theory
in great generality, similar to \cite{spitzweck.herm}, which is partly in progress.
A geometric description, known for regular bases where $2$ is invertible at least
since Girja Tripathi's thesis, is missing from this impressive
list of properties.
Theorem~\ref{thm:chn-kumar} below provides that the two motivic spectra are
equivalent over any quasi-compact quasi-separated scheme.
As one consequence, work by the second author with Kolderup and {\O}stv{\ae}r
extends slice and coefficient computations to hermitian $K$-theory over
Dedekind domains, including $\Z$ \cite{kro.very}.
Applications to motivic stable homotopy groups of spheres will be pursued elsewhere.

To prove the equivalence we use two main results.
As a first result, the collaboration of Panin and Walter \cite{panin-walter.quat-borel, panin-walter.msp}
provides a description of $\KQ^{\ast+(\star)}(\HGr_\Z )$.
Using this description allows us to construct a map
\[ \phi\colon \Z\times\HGr_\Z\to \Omega^{\infty+(\infty)}\Sigma^{2+(2)}\KQ_{\Z}\] of motivic spaces.
Its source features $\HGr_\Z$, the quaternionic Grassmannian defined over $\Z$.
The target of $\phi$ is the infinite $\P^1$-loop space of the suitably shifted hermitian
$K$-theory spectrum over $\Z$.
We then show this is compatible with the structure maps of $\KQ^\prime_\Z$ and $\KQ_\Z$.
The resulting map $\phi\colon \KQ^\prime_\Z\to \KQ_\Z$ of motivic spectra is then an equivalence.
To prove this we use as a second main result, the Grothendieck-Witt groups defined in \cite{9.III} agree with the classical hermitian K-groups of Karoubi in positive degrees for polynomial rings over fields. This fact is proved in a recent preprint of Schlichting \cite{sch.sym.v.gensym}, but it follows rather directly from the results of \cite{9.III} as described below.

\subsubsection*{Notation}
Let $S$ be a qcqs (quasi-compact and quasi-separated) base scheme.
We denote by
$S^{t,w}=S^{t-w}\wedge \mathbb{G}_{m}^{w} = S^{t-w+(w)}$ the motivic sphere of topological degree $t$ and weight $w$.
The corresponding suspension functor on the stable motivic homotopy category $\SH(S)$ is denoted 
by $\Sigma^{t,w}=\Sigma^{t-w+(w)}$.
For any pointed motivic space $X$ with associated motivic suspension spectrum
$\Sigma^{\infty+(\infty)}X$ and any motivic spectrum $\E\in \SH(S)$, set
\[ \E^{t+(w)}(X)= [\Sigma^{\infty+(\infty)} X, \Sigma^{t+(w)}\E]_{\SH(S)} \]
as the corresponding group of homomorphisms in $\SH(S)$.
The notation ``$\Sigma^{\infty+(\infty)}$'' may be left out for simplicity.
In the particular case of motivic spheres, one associates to any motivic spectrum $\E\in\SH(S)$ the bigraded motivic homotopy group 
\[
\pi_{t,w}\E=
\pi_{t-w+(w)}\E=
[S^{t,w},\E]_{\SH(S)}
\]
for $t,w\in \Z$.
For example, 
the hermitian $K$-groups of $S$ can be defined as $\KQ_{t,w}(S)=\pi_{t,w}\KQ_S$.
In the literature, 
the notation $\GW^w_t$ is also used to denote hermitian $K$-groups; 
these are related to $\KQ_{t,w}$ via the indexing conventions visible in the identifications
\[
\GW^{w}_{t}(S)\cong [S^{t,0}, \Sigma^{(w)}\KQ_S]=\pi_{t-2w,-w}\KQ_S=\pi_{t-w-(w)}\KQ_S.
\]
For an affine scheme $X=\Spec(R)$, we use both $\KSp(X)$ and $\KSp(R)$ to denote the K-theory spectrum of the category of non-degenerate alternating (or, in the language of \cite{9.I}, genuine $(-1)$-even) forms on $X$.
As a consequence, $\KSp$ denotes the corresponding spectrum valued presheaf $X\mapsto \KSp(X)$, with $\KSp_n(X)$ and $\KSp_n(R)$ short for the homotopy groups $\pi_n \KSp(X)$.

\subsection*{Acknowledgements}
Our work was generously supported by the DFG
and the RCN Project no.~312472 ``Equations in Motivic Homotopy''.
We thank Daniel Marlowe and Marco Schlichting for helpful comments. 

\section{The two candidates}
\label{aandh}

Recent work achieved a satisfying status for Karoubi's hermitian $K$-theory (also known as higher Grothendieck-Witt theory), which until recently was available only over base schemes in which $2$ is invertible.

\begin{theorem}[\cite{chn}]\label{thm:chn}
Let $S$ be a qcqs scheme. 
There exists a motivic $E_{\infty}$ ring spectrum $\KQ\in \SH(S)$ with the following properties.
\begin{enumerate}
    \item There exists an element $\alpha:\Sigma^{4+(4)}\unit\to \KQ$ such that multiplication with $\alpha$ is an equivalence $\KQ \simeq \Sigma^{4+(4)}\KQ$.
    \item\label{item:wood-seq} There exists a ring map $\Forg\colon \KQ_S\to \KGL_S$ and a $\KQ_S$-module map $\Hyper\colon \KGL_S\to \KQ_S$ such that \[ \Sigma^{(1)} \KQ \xrightarrow{\eta\smash \KQ} \KQ \xrightarrow{\Forg} \KGL \xrightarrow{\Sigma^{1+(1)}\Hyper \circ \beta} \Sigma^{1+(1)}\KQ \] is a cofiber sequence.
    \item For every morphism $f\colon R\to S$ there exists a canonical equivalence $f^\ast \KQ_S\simeq \KQ_R$.
    \item For every closed embedding $i\colon R\to S$ of regular schemes of codimension $c$ and normal bundle $Ni$, the purity transformation $i^\ast\KQ_S\to \Th(Ni)\smash i^!\KQ_S$ is an equivalence.
    \item For every vector bundle $V\to S$ of rank $r$, there exists a Thom equivalence $\KQ_S\simeq \Th(V)\smash \KQ_S(\det V[-r])$.\footnote{The shifted line bundle in parentheses denotes an appropriate twist of $\KQ_S$.}
    \item If $S$ is regular and $2$ is invertible on $S$, $\KQ_S$ represents Karoubi's hermitian $K$-groups over $S$.
\end{enumerate}
\end{theorem}

\begin{proof} 
 The first statement is \cite[Remark 8.1.2]{chn}; note that $\Forg(\alpha)=\beta^4$.
 The second statement follows from \cite[Corollary 8.1.7]{chn} and \cite[Remark 8.1.8]{chn}.
 The absoluteness statement 3 is \cite[Proposition 8.2.1]{chn}.
 Statement 4 follows from \cite[Theorem 8.4.2]{chn} and statement 5 can be deduced from \cite[Proposition 8.3.1]{chn}.
 Statement 6 follows from \cite[Proposition B.2.2]{9.II} and \cite[Corollary 8.1.5]{chn}.
 \end{proof}

Theorem~\ref{thm:chn} relies on joint work with six further colleagues, which is partly in progress, and partly available \cite{9.I, 9.II, 9.III}.
The long list of properties is still not sufficient to transfer the known arguments for slice computations of hermitian $K$-theory over fields of characteristic zero to the more general setting.
Some geometry for an effectivity statement is required.
To state the geometric description, let $\HGr_S$ denote the infinite quaternionic Grassmannian over $S$ constructed in \cite{panin-walter.quat-borel}. 

\begin{theorem}[\cite{kumar.thesis}]\label{thm:kumar}
Let $S$ be a qcqs scheme. 
There exists a motivic spectrum $\KQ_S^\prime\in \SH(S)$ with the following properties.
\begin{enumerate}
    \item[$6^\prime$] If $S$ is regular and $2$ is invertible on $S$, $\KQ_S$ represents Karoubi's hermitian $K$-groups over $S$.
    \item[$7^\prime$] There exists a structure map $\Sigma^{4+(4)}\Z\times \HGr_S \to \Z\times \HGr_S$ of pointed motivic spaces over $S$ such that $\Sigma^{2+(2)}\KQ_S^\prime$ is the motivic spectrum associated to the Bousfield-Friedlander type $\Sigma^{4+(4)}$-spectrum $(\Z\times \HGr_S,\Z\times \HGr_S,\dotsc)$ obtained from this structure map.
\end{enumerate}
In particular, $\KQ^\prime_S$ is cellular.
\end{theorem}

The cellularity basically follows as in \cite{rso.cell}.
It is worth remarking that other work towards hermitian $K$-theory over schemes in which $2$ is not necessarily invertible existed \cite{zbMATH07124513, spitzweck.herm}.
Also the ring structure has been addressed already in \cite{lopez.thesis}, but is not
easily transferrable to the geometric version $\KQ^\prime$.

\begin{theorem}\label{thm:chn-kumar}
Let $S$ be a qcqs scheme.
There exists an equivalence $\KQ_S^\prime \to \KQ_S$.
\end{theorem}

\begin{proof}
  Since both motivic spectra pull back by Theorem~\ref{thm:chn} and Theorem~\ref{thm:kumar}, it suffices to construct an equivalence $f\colon \KQ^\prime_\Z\to \KQ_\Z$.
  Its main ingredient, a map
  \[ \phi\colon \Z\times \HGr_\Z\xrightarrow{} \Omega^{\infty+(\infty)}\Sigma^{2+(2)}\KQ_\Z\]
  of motivic spaces over $\Z$, is given in Lemma~\ref{lem:KQ-BSp} below.
  Via the explicit periodicity equivalence $\Sigma^{4+(4)}\KQ_\Z\simeq \KQ_\Z$, there
  results a map  
  \[ \phi_n\colon \Z\times \HGr_\Z\xrightarrow{\simeq} \Omega^{\infty+(\infty)}\Sigma^{4n+2+(4n+2)}\KQ_\Z\]
  of motivic spaces over $\Z$ for every natural number $n$. 
  That these maps define a map of motivic spectra requires the diagram
  \begin{equation}\label{eq:structure-map}
    \begin{tikzcd}
      \Sigma^{4+(4)} \Z\times \HGr_\Z \ar[r,"\Sigma^{4+(4)}\phi_n"] \ar[d,"\alpha"] &
      \Sigma^{4+(4)}\Omega^{\infty+(\infty)}\Sigma^{4n+2+(4n+2)}\KQ_\Z \ar[d] \\
      \Z\times \HGr_\Z \ar[r,"\phi_{n+1}"] & \Omega^{\infty+(\infty)}\Sigma^{4n+6+(4n+6)}\KQ_\Z
    \end{tikzcd}
  \end{equation}
  where the vertical arrows depict the respective structure maps of the motivic spectra 
  to commute for every $n\in \N$.
  Using the adjunction $(\Sigma^{4+(4)},\Omega^{4+(4)})$, each of the two paths in diagram~(\ref{eq:structure-map}) defines a homotopy class
  $\Z\times \HGr_\Z\to \Omega^{\infty+(\infty)}\Sigma^{4n+2+(4n+2)}\KQ_\Z$.
  To conclude that these coincide, observe that the Thom isomorphism \cite[Proposition 8.3.1]{chn} for $\KQ_{\Z}$ provides, 
  for every special linear vector bundle $V\to X$ of rank $r$ with $X$ smooth over $\Spec(\Z)$, 
  a class
  \[ \mathrm{th}(V\to X)\in \KQ^{r+(r)}(\Sigma^{\infty+(\infty)}\Th(V))\cong \KQ^{0+(0)}(X)\]
  corresponding to the unit in the latter ring.
  It is straightforward to check that these classes constitute a normalized $\mathrm{SL}$-orientation on $\KQ_{\Z}$ in the sense of \cite[Section~5]{panin-walter.msp}.
  Hence the motivic spectrum $\KQ_{\Z}$ obtains a symplectic orientation, 
  that is, a map $\mathsf{MSp}_{\Z}\to \KQ_{\Z}$ which is a homomorphism of commutative monoids in the 
  homotopy category; see \cite[Theorem 1.1]{panin-walter.msp}.
  
  The theory of Borel classes associated with the symplectic orientation on $\KQ_S$ supplies an isomorphism 
  \[ \KQ_{S}^{s+(w)}(\HGr)\cong \bigl(\KQ^{\ast+(\star)}_{S}(S)\llbracket b_1,b_2,\dotsc\rrbracket\bigr)^{s+(w)}\]
  where $b_j\in \KQ_S^{2j+(2j)}(\HGr_S)$, see \cite[Theorem 11.4]{panin-walter.quat-borel}. This holds in particular for $S=\Spec(\Z)$.
  Thus the two homotopy classes under consideration can be identified with
  elements in the group
  \[ \KQ_{\Z}^{4n+2+(4n+2)}(\HGr_\Z)\cong \bigl(\KQ^{\ast+(\star)}({\Z})\llbracket b_1,b_2,\dotsc\rrbracket\bigr)^{4n+2+(4n+2)}\]
  where $b_j\in \KQ_\Z^{2j+(2j)}(\HGr_\Z)$.
  It follows that the respective power series coefficients for both of the
  homotopy classes reside in the groups $\KQ^{2m+(2m)}(\Z)$, which by
  periodicity reduce to the groups $\KQ^{0+(0)}(\Z)$ and $\KQ^{2+(2)}(\Z)$.
  Since the inclusion $\Z\hookrightarrow \R$ induces an injection (even an
  isomorphism) on these groups, the commutativity of diagram~(\ref{eq:structure-map})
  up to homotopy follows from the respective commutativity over $\R$.
  The latter follows from the fact that $\KQ_\R$ and $\KQ^\prime_\R$ are equivalent
  to the motivic spectrum representing Karoubi's hermitian $K$-groups of smooth
  $\R$-varieties, as listed in Theorem~\ref{thm:chn} and Theorem~\ref{thm:kumar}.
  Since $2$ is invertible in $\R$, $\phi_n$ is mapped to $\tau_{4n+2}$ in Section 12 of \cite{panin-walter.bo} under the equivalence $T^2\simeq \HP^{1+}$.
  As a consequence, diagram~(\ref{eq:structure-map}) commutes up to homotopy
  for every $n$.

  To provide an actual map of motivic spectra, commuting up to homotopy does not suffice.
  However, the argument for the commutativity up to homotopy shows that the 
  sequence $(\phi_n)$ is an element
  in the group $\lim\limits_{n\to \infty} \KQ^{4n+2+(4n+2)}(\Z\times \HGr)$.
  This group sits in a short exact sequence
  \[ 0 \to {\lim_n}^1\KQ^{4n+2+(4n+2)}(\Z\times \HGr) \to [\KQ^\prime,\KQ]_{\SH(\Z)} \to \lim\limits_{n\to \infty}\KQ^{4n+2+(4n+2)}(\Z\times \HGr) \to 0 \]
  by \cite[Lemma 1.3.3]{ppr}, where the middle group represents the set of
  maps from $\KQ^\prime$ to $\KQ$ in the motivic stable homotopy category $\SH(\Z)$.
  Hence there exists a lift $f\colon \KQ^\prime_\Z \to \KQ_\Z$ of the sequence $(\phi_n)$,
  that is, a map $f$ of motivic spectra whose level $f_n\colon \Z\times \HGr \to \KQ_{4n+2}$ coincides with $\phi_n$ up to homotopy.
  (In this situation, the $\lim^1$-term vanishes, which is not relevant for the existence.)
  To prove that this map is an equivalence, consider the diagram
  \begin{equation}\label{eq:uivalence} \begin{tikzcd}
      \Z\times \HGr_\Z \ar[r] \ar[d,"f_0"] & \Omega^{4+(4)}\Z\times \HGr_\Z \ar[r] \ar[d,"\Omega^{4+(4)}f_1"] & \dotsc \ar[r]  & \Omega^{\infty +(\infty)} \Sigma^{2+(2)}\KQ^\prime _\Z \ar[d] \\
      \Omega^{\infty+(\infty)}\Sigma^{2+(2)}\KQ_\Z \ar[r] & \Omega^{4+\infty+(4+\infty)}\Sigma^{6+(6)}\KQ_\Z \ar[r] & \dotsc \ar[r] & \Omega^{\infty+(\infty)}\Sigma^{2+(2)}\KQ_\Z
    \end{tikzcd} 
  \end{equation}
  whose horizontal maps are induced by the structure maps of the motivic spectra $\KQ^\prime_\Z$ and $\KQ_\Z$, respectively, and which ends with the respective colimit on the right hand side.
  Diagram~(\ref{eq:uivalence}) commutes, since so does diagram~(\ref{eq:structure-map}). 
  The bottom horizontal maps in diagram~(\ref{eq:uivalence}) are all equivalences of motivic spaces by Theorem~\ref{thm:chn}.
  Lemma~\ref{lem:KQ-BSp} implies that the vertical map $\Omega^{4n+(4n)}(f_n)$ is an equivalence of motivic spaces for $n>0$.
  Hence the map induced on the colimit is an equivalence of motivic spaces.
  It follows that the map $f\colon \KQ^\prime_\Z\to \KQ_\Z$ is an equivalence.
\end{proof}

\begin{lemma}\label{lem:KQ-BSp}
  Let $\HGr_\Z$ be the infinite quaternionic Grassmannian over $\Z$.
  There exists a map
  \[ \phi\colon \Z\times \HGr_\Z \to \Omega^{\infty+(\infty)}\Sigma^{2+(2)}\KQ_\Z\]
  of motivic spaces inducing an injection on $\pi_0$ and isomorphisms on $\pi_n$
  for every $0<n\in \N$. 
\end{lemma}

\begin{proof}
  To construct $\phi$, the theory of Borel classes associated with the symplectic orientation on $\KQ_S$, as explained in the proof of Theorem~\ref{thm:chn-kumar}, will be used.
  This results in an isomorphism 
  \[ \KQ_{S}^{s+(w)}(\HGr)\cong \bigl(\KQ^{\ast+(\star)}_{S}(S)\llbracket b_1,b_2,\dotsc\rrbracket\bigr)^{s+(w)}\]
  where $b_j\in \KQ_S^{2j+(2j)}(\HGr_S)$, see \cite[Theorem 11.4]{panin-walter.quat-borel}. This holds in particular for $S=\Spec(\Z)$. By \cite[Section 3.2]{9.III} and \cite[Corollary 8.1.5]{chn} there is an inclusion of groups
  \[ \Z \oplus\Z\cong \KSp_0(\Z)\oplus\Z\{b_1\}\to \KQ_{\Z}^{2+(2)}(\HGr) \]  
  where $\KQ^{2+(2)}_{\Z}(\Z)\cong \KSp_0(\Z)\cong \Z$ is generated by the hyperbolic form of rank $2$ over $\Z$.
  Let $\phi\colon \Z\times \HGr_\Z \to \Omega^{\infty+(\infty)}\Sigma^{2+(2)}\KQ_\Z $ be the morphism given by the sequence of elements $(i+b_1)_{i\in \Z}$  in $\KSp_0(\Z)\oplus\Z \{b_1\}$.
  This provides a unique element in $[\Z\times\HGr_\Z, \Omega^{\infty+(\infty)}\Sigma^{2+(2)}\KQ_\Z ]$, the set of maps in the unstable $\A^1$-homotopy category $\mathbf{Ho}(\Z)$ of $\Z$.

  Having produced $\phi\colon \Z\times \HGr_{\Z}\to \Omega^{\infty+(\infty)}\Sigma^{2+(2)}\KQ_{\Z}$, it remains to prove the assertion on homotopy sheaves.
  For this purpose, consider the map of homotopy presheaves $X\mapsto [\Sigma^n X_+,\phi]$ inducing the map on homotopy sheaves.
  The domain $[\Sigma^n X_+,\Z\times \HGr_{\Z}]$ of this map is naturally identified with $\pi_n\Sing^{\A^1}\Omega^\infty\KSp(X)$ in \cite[Theorem 3.5]{kumar}.
  The map $\phi$ is then induced by the canonical map 
  \[ \phi(X)\colon \Sing^{\A^1}\KSp(X) \to \Sing^{\A^1}\Omega^{(\infty)}\Sigma^{2+(2)}\KQ_{\Z}(X)\simeq \Omega^{(\infty)}\Sigma^{2+(2)}\KQ_{\Z}(X)\]
  of spectra comparing genuine $(-1)$-even to (homotopy) $(-1)$-symmetric forms, the latter being $\A^1$-invariant \cite[Theorem 6.3.1]{chn}.
  Hence $\phi$ arises as the evaluation of a map of presheaves with values in spectra.
  Let $i\colon \Spec(\mathbb{F}_2)\hookrightarrow \Spec(\Z)$ denote the closed embedding with open complement $j\colon \Spec(\Z[2^{-1}])\hookrightarrow \Spec(Z)$.
  The $S^1$-stabilization
  \[ j_\sharp j^\ast \to \mathrm{Id}_{\mathbf{Ho}(\Z)} \to i_\ast i^\ast\to \Sigma j_\sharp j^\ast\] 
  of the unstable localization cofiber sequence from \cite{Morel-Voevodsky}
  applies to give a natural transformation of long exact sequences
  \[ \dotsm \to [\Sigma^nX_+,j_\sharp j^\ast \phi] \to[\Sigma^nX_+,\phi] \to[\Sigma^nX_+,i_\ast i^\ast \phi] \to [\Sigma^{n-1}X_+,j_\sharp j^\ast \phi] \to \dotsm. \]
  By parts 3 and 5 of Theorem \ref{thm:chn}, for any $f\colon R\to S$ and any vector bundle $V\to S$ the square
  \[
  \begin{tikzcd}
    f^*\KQ_{S}\arrow[r,"\simeq" ]\arrow[d,"\simeq"] &\KQ_R\arrow[d,"\simeq"]\\
    f^*\Th(V)\smash f^*\KQ_{S}(\det V[-r])\arrow[r,"\simeq"] &\Th(f^*V)\smash \KQ_{R}(\det f^*V[-r])
  \end{tikzcd} \] 
  commutes.
  This implies that the corresponding $\SL$-orientation is stable under base change.
  In particular, the Borel classes are compatible with pullbacks.
  Our definition of $\phi$ is the same as Panin and Walter's identification $\Z\times\HGr\simeq \Omega^\infty\KSp$ when $2$ is invertible; see \cite[Section 11]{panin-walter.bo}.
  The compatibility of $\phi$ guaranteed above implies that $j^\ast\phi$ is an equivalence.
  Hence so is $j_\sharp j^\ast \phi$.
  It follows that the required connectivity for $\phi$ holds if it holds for $i^\ast \phi$. 
  Having placed ourselves in the $S^1$-stable $\A^1$-homotopy category over the
  perfect field $\FF_2$, these
  $\A^1$-homotopy sheaves are strictly $\A^1$-invariant by a theorem of Morel \cite[Corollary~6.2.9]{morel.stable-conn}.
  For these, it suffices to evaluate on fields $F$ over $\FF_2$.
  The map
  \[ i^\ast\phi(F)\colon \Sing^{\A^1}\KSp(F) \to \Sing^{\A^1}\Omega^{(\infty)}\Sigma^{2+(2)}\KQ_{\FF_2}(F)\]
  is induced by the map 
  \[ i^\ast\phi(\Delta^d_F)\colon \KSp(\Delta^d_F) \to \Omega^{(\infty)}\Sigma^{2+(2)}\KQ_{\FF_2}(\Delta^d_F)\]
  where $\Delta^d_F=\Spec F[t_0,\dotsc,t_d]/(\sum t_j=1)$.
  Its target is equivalent to $\Omega^{(\infty)}\Sigma^{2+(2)}\KQ_{\FF_2}(F)$ by
  $\A^1$-invariance for $\KQ$. 
  Since $\Delta^d_F$ has  Krull dimension $d$, \cite[Corollary 1.3.9]{9.III} for $r=1$ implies that $\pi_ni^\ast \phi(\Delta^d_F)$ is injective for $n\geq d$ and bijective for $n>d$.
  However, the arguments of \cite[Section 1.3]{9.III} give more in this special situation.
  The relevant base ring $R$ is a polynomial ring in $d$ variables over the field $F$, with
  trivial involution and canonical duality.
  Every finitely generated projective $R$-module is free, by the solution of Serre's conjecture \cite{quillen.proj,suslin.sl}.
  Hence the duality on the resulting derived category of finitely generated
  projective $R$-modules interacts with the $t$-structure in the same way as in the
  case of the field $F$; see \cite[p. 27, Remark 1.18]{9.III}.
  Thus $\pi_ni^\ast \phi(\Delta^d_F)$ is injective for $n\geq 0$ and bijective for $n>0$,
  for every $d$. The same property then holds after realization, which shows that
  \[ i^\ast\phi(F)\colon \Sing^{\A^1}\KSp(F) \to \Sing^{\A^1}\Omega^{(\infty)}\Sigma^{2+(2)}\KQ_{\FF_2}(F)\]
  has the desired connectivity property. This concludes the proof. 
\end{proof}

\begin{remark}
  The map $\pi_0i^*\phi$ above fails to be surjective.
  For $F$ a perfect field of characteristic $2$, genuine $(1)$-symmetric and $(-1)$-symmetric forms are equivalent, and every unit is a square.
  Thus
  \[\pi_0\Omega^{\infty+(\infty)}\Sigma^{2+(2)}\KQ_{\FF_2}(F)\cong \GW_0^0(F)\cong \Z\]
  and the morphism $\pi_0i^\ast\phi(F)\colon \KSp_0(F)\to \GW_0^0(F)$ is the map $2\colon \Z\to\Z$; see \cite[Remark 5.8]{hjny}.
	
\end{remark}	
\newcommand{\etalchar}[1]{$^{#1}$}

\end{document}